\documentclass[11pt]{article}

\usepackage{amssymb}
\usepackage{amsmath}
\usepackage{amsthm}
\usepackage{float}
\usepackage[T1]{fontenc}
\usepackage[utf8]{inputenc}
\usepackage{overpic}
\usepackage{geometry}
\usepackage{caption}
\usepackage{graphicx}
\usepackage{dsfont}

\usepackage{enumerate}

\usepackage{mathrsfs}

\newcommand{\R}{\mathbb{R}}
\newcommand{\bC}{\mathbb{C}}

\newcommand{\G}{\Gamma}
\newcommand{\g}{\gamma}

\newcommand{\cM}{\mathcal M}

\newcommand{\re}{\mathrm{Re}\,}

\newcommand{\wt}{\widetilde}

\newtheorem{theorem}{Theorem}[section]
\newtheorem{definition}[theorem]{Definition}
\newtheorem{lemma}[theorem]{Lemma}
\newtheorem{proposition}[theorem]{Proposition}
\newtheorem{corollary}[theorem]{Corollary}

\DeclareMathOperator{\Mod}{\mathscr{M}}
\DeclareMathOperator{\Isom}{Isom}

\author{Olivier Glorieux and Andrew Yarmola}

\title{Random triangles on flat tori}

\begin{document}
\maketitle

\begin{abstract}
Inspired by classical puzzles in geometry that ask about probabilities of geometric phenomena, we give an explicit formula for the probability that a random triangle on a flat torus is homotopically trivial. Our main tool for this computation involves reducing the problem to new invariant of measurable sets in the plane that is unchanged under area-preserving affine transformations. Our result show that this probability is minimized at all rectangular tori and maximized at the regular hexagonal torus.
\end{abstract}

\section{Introduction}
A classical problem, recorded as problem 58 in a 1893 book of puzzles by Charles Dodgson \cite{Dodgson}, is to determine the probability that a triangle on the plane is obtuse. Since the plane has no finite measure invariant by translation, this problem is not well defined and several answers can be given. For triangles in finite area regions, this questions has been studied by Kendall \cite{Kendall},  Guy \cite{Guy} and Portnoy \cite{Portnoy}, among others. On the sphere the problem has been treated by Cooper-Eum-Li \cite{LiCooperEum} and Isokawa look at similar problems on the hyperbolic plane \cite{Isokawa}. Our interest lies in extending such probabilistic puzzles to finite area surfaces.\\

For us, a triangle is a choice of three points of a surface with vertices connected by shortest geodesic segments. The first thing to note is that such triangles fall into two classes: homotopically trivial and homotopically non-trivial. With this dichotomy in mind, we ask the following question for triangles on surfaces:
\begin{itemize}
\item What is the probability of a random triangle being homotopically trivial on a surface?
\end{itemize}

In this note, we answer this question for flat tori. Let $\Mod_1$ be the moduli space of flat structures of the torus $T^2$. We identify $\Mod_1$ with the set of points in the modular domain 
$$\cM = \{a+i b \in \bC \mid  a \in (-1/2,1/2], \; b > \sqrt{1-a^2}\text{ or } b = \sqrt{1-a^2} \text{ and } a \in [0,1/2] \}.$$
For $\tau \in \cM$, we let $\Sigma_\tau$ be the flat structure $\R^2/ \langle 1, \tau \rangle$ in $\Mod_1$. Define $P(\tau)$ to be the probability that three points, sampled uniformly on $\Sigma_\tau$, give rise to a homotopically-trivial triangle by connecting them by shortest geodesic segments. Note, since shortest geodesic segments between points are unique away from a set of measure zero, $P(\tau)$ is well-defined. Our main result gives $P(\tau)$ as a rational function of $a, b$ where $a + b i = \tau \in \cM$.
\begin{theorem}\label{thm:flat_torus} For $\tau \in \cM$ one has
$$P(\tau) = \frac{9}{16} + \frac{3|a|^2}{8b^2} - \frac{|a|^3}{2b^2} - \frac{|a|^3}{2b^4} + \frac{17 |a|^4}{16b^4} - \frac{|a|^5}{2b^4}$$ 
\end{theorem}

\begin{corollary} $$\frac{1}{vol(\Mod_1)}\int_{\tau \in \Mod_1} P(\tau) \, d \tau = \frac{1}{20} \left(13 - \frac{3\,\sqrt{3}}{\pi}\right)$$
where $d\tau$ is the Teichm\"uller (i.e. hyperbolic) metric on $\Mod_1$ and so is the volume.
\end{corollary}

\begin{corollary} $P(\tau) = 9/16$ if and only if the two shortest closed geodesics on $\Sigma_\tau$ meet at a right angle. Further, $P(\tau) = 7/12$ if and only if $\Isom(\Sigma_\tau)$ contains $D_6$ as a subgroup.
\end{corollary}

Our main tool for this computation involves reducing the problem to computing a new invariant for measurable sets in the plane that is preserved under the action of $\R^2 \rtimes \mathrm{SL}_2(\R)$. This reduction arrises by rephrasing the question in terms of the overlap of the Dirichlet domains of the sampled points. Even though the observation about Dirichlet domains extends to surfaces of higher genus with a hyperbolic or even Riemannian metric, the process of computing an exact value in that setting seems rather difficult. For example, for the complete finite-area hyperbolic metric on the thrice punctured sphere, the computation of $P(\tau)$ begins to involve elliptic functions and their integrals. The asymptotic behavior of $P(\tau)$ over the moduli space of higher genus surfaces is more amenable to an analysis, which we address in a forthcoming paper \cite{forthcoming}. 

\section{Reduction to the overlap invariant}

In this section, we reduce the computation of $P(\tau)$ to that of an invariant of triples of measurable sets in $\R^2$ that is preserved under  the action of $\R^2 \rtimes \mathrm{SL}_2(\R)$.
\begin{definition}
Let $A\subset \R^2$ be a subset. For $w \in \R^2$, define 
$$A_w := \{ x+ w\,|\, x\in A\}.$$
Given measurable $A,B,C\subset \R^2$ , consider the function 
$$F(A,B,C) := \int_{a \in A} area(B\cap C_a) \, da.$$
\end{definition}

For $\tau \in \cM$, we have $\Sigma_\tau = \R^2/ \Gamma_\tau$ where $\Gamma_\tau = \langle 1, \tau\rangle$. Recall that the Dirichlet domain of a point $x \in \R^2$ is the set $D_{\tau, x} = \{ y \in \R^2 \mid d(x,y) \leq d(x, \gamma y) \text{ for all } \gamma \in \Gamma_\tau\}$. Since $\Gamma_\tau$ is a subgroup of translations in the plane, $D_{\tau, x_1}$ and $D_{\tau, x_2}$ are translates of each other for any $x_1, x_2 \in \R^2$. It therefore makes sense to fix the origin $o = (0,0) \in \R^2$ and let $D_\tau = D_{\tau, o}$. A key property of Dirichlet domains is the fact that the distance between $y\in D_{\tau,x}$ and $x$ on $\R^2$ is equal to the distance between $x+\G_t$ and $y+\G_t$ on $\Sigma_\tau$. In particular the geodesic line between $x$ and $y\in D_{\tau,x}$ in $\R^2$, projects to the minimizing geodesic on the torus. 

\begin{proposition}
$P(\tau) = F(D_\tau,D_\tau,D_\tau)/area(D_\tau)^2$
\end{proposition}
\begin{proof}
Pick $3$ points $\{x_1,x_2, x_3\}$ on the flat torus $\Sigma_\tau$. Connecting these by shortest geodesics gives a triangle. Lifting to the universal cover, we can suppose that $\wt{x}_1 = o$ is the center of $D_\tau$ and choose $\wt{x}_2, \wt{x}_3 \in D_\tau$. Now, by lifting the geodesic from $x_1$ to $x_2$, we see that the triangle formed by $\{x_1,x_2,x_3\}$ on $\Sigma_\tau$ is homotopically trivial if and only if $\wt{x}_3$ is contained in the Dirichlet domain centered at $\wt{x}_2$. Thus, given $\wt{x}_2$, the set of $\wt{x}_3$ that give a homotopically trivial triangle is precisely $D_\tau \cap D_{\tau, \wt{x}_2} = D_\tau \cap (D_\tau)_{\wt{x}_2}$. Integrating this over all possible choices of for $\wt{x}_2 \in D_\tau$ gives $F(D_\tau,D_\tau,D_\tau)$. Lastly, normalizing to unit area gives the scaling factor of $1/area(D_\tau)^2$ and the desired result.
\end{proof}

To compute $P(\tau)$, we analyze the behavior of $F$ and the shape of $D_\tau$.

\begin{proposition}[Properties of $F$]\label{prop:prop_of_F}\hspace*{0.1in}
\begin{enumerate}
\item $F$ is invariant under $\R^2 \rtimes \mathrm{SL}_2(\R)$ acting diagonally by affine transformations.
\item $F(t \, A, t \, B, t \, C) = t^4 \, F(A,B,C)$ where $t X$ denotes scaling $X$ by $t \in \R^+$.
\item $F$ is additive with respect to disjoint union in $A, B,$ and $C$.
\item $F(Q,Q,Q)=9/16$ where $Q = [-1/2,1/2] \times [-1/2,1/2]$ is the unit square. 
 \item $F(\R^2,X,X)= area(X)^2$ for all measurable sets $X$.
\end{enumerate}
\end{proposition}
\begin{proof}  Statements (1), (2), and (3) follow from the definition of $F$. For (4) we can use the symmetry of the unit square and compute directly in each quadrant as follows.
\[F(Q,Q,Q) =  4\int_{0}^{1/2}\int_{0}^{1/2} (1-x)(1-y)dxdy = \frac{9}{16}.\]		
For (5), we observe that $F$ is just the integral of the $\R^2$ convolution of indicator functions. 
\begin{eqnarray*}
F(\R^2,X,X) &=& \int_{w\in \R^2} area( X\cap X_w) \, dw\\
					&=& \int_{w \in \R^2} \int_{s\in \R^2}\mathds{1}_X(s)\mathds{1}_X(s+w) \,ds dw\\
					&=&  \int_{s\in \R^2}\left(\mathds{1}_X(s)\int_{x\in \R^2}\mathds{1}_X(s+w) \,dw \right) \,ds\\
					&=&  \int_{s\in \R^2}\mathds{1}_X(s) \,ds \int_{w\in \R^2}\mathds{1}_{X_{-s}}(w) \, dw \\
					&=& area(X)^2
\end{eqnarray*}

\end{proof}


\section{The shape of Dirichlet domains}

Whenever $a = \re \tau \neq 0$ the shape of $D_\tau$ is a hexagon with parallel sides. Otherwise, it is a rectangle. It is clear that $D_\tau$ and $D_{-\overline{\tau}}$ are mirror images of each other across the vertical axis. We will therefore assume that $a \in [0,1/2]$ for the rest of this section.

\begin{proposition}
Let $\tau= a+ib \in \cM$ with $a \in [0,1/2]$. Then  
$D_\tau$ is given by the convex hull of the following six points  
$A = (1/2,\alpha)$, $B=(a-1/2,\beta)$, $C=(-1/2,\alpha)$, $A'=-A$, $B'=-B$, and $C'=-C$, where 
$\displaystyle{\alpha = \frac{b^2+a^2-a}{2b}}$ and $\displaystyle{\beta=  \frac{b^2 - a^2 +a}{2b}}$.
\end{proposition}
\begin{proof}
The Dirichlet domain is given by the intersection of the half-planes 
$$ H_\gamma =\{x \in \R^2 \, |\, d(x, o) \leq d(\g x, o)\} \text{ for } \g \in \Gamma_\tau.$$ 
Since $\Gamma_\tau = \langle 1, \tau \rangle$, most intersections are redundant. Indeed, a simple computation show that since $a \in [0,1/2]$, the Dirichlet domain is  given by the intersection of only $6$ half-planes corresponding to $\g = (1,0), (a,b), (a-1,b)$ and their inverses.

Therefore $D_\tau$ is bounded by the $6$ lines whose equations are 
$$L_{(1,0)} := \{x=1/2\}$$
$$L_{(a,b)} := \{ax+by = \frac{a^2+b^2}{2}\}$$
$$L_{(a-1,b)} := \{(a-1)x+by = \frac{(a-1)^2+b^2}{2}\},$$
and their reflections around the origin. 

It is then easy to compute that 
\begin{align*}
& L_{(1,0)}\cap L_{(a,b)} = (1/2,\alpha)=A  \\
& L_{(a,b)}\cap L_{(a-1,b)} = (a-1/2,\beta)=B  \\
&  L_{(-1,0)}\cap L_{(a,b)} = (-1/2, \alpha) = C.
\end{align*}
\end{proof}

We will apply an element of $\mathrm{GL}_2(\R)$ to map $D_\tau$  to a particular hexagon family. Let $H(s,t)$ be the hexagon obtained from the square $[-1/2, 1/2] \times [-1/2,1/2]$ by removing the two right angled triangles $T$ and $-T$ from the corners. The edge lengths of $T$ and $-T$ will be $s$ and $t$ as seen in Figure \ref{fig:H_st}.

\begin{figure}[hbt]
    \begin{center}
\begin{overpic}[scale=1]{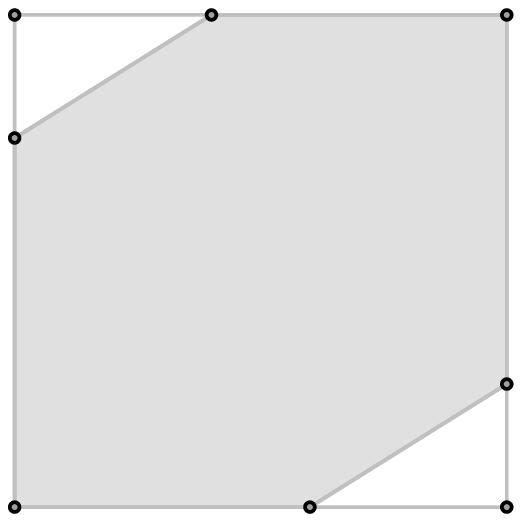}
\put(20,97){$s$}
\put(0,82){$t$}
\put(40,45){$H(s,t)$}
\put(13,84){$T$}
\put(77,10){$-T$}
\end{overpic}
\end{center}
\caption{The $H(s,t)$ hexagon inside the unit square $Q$.}
\label{fig:H_st}
\end{figure}

\begin{proposition}\label{prop:st_shape}
Let $\tau= a+ib \in \cM$ with $a \in [0,1/2]$. Then $D_\tau$ is $\mathrm{GL}_2(\R)$ conjugate to $H(s,t)$ for  $s= a$ and $t= \frac{a}{a^2+b^2}$.
\end{proposition}
\begin{proof}
The element 
$$\g=\begin{pmatrix}
    1&0  \\
    \frac{a}{a^2+b^2} & \frac{b}{a^2+b^2}
\end{pmatrix}$$ sends $D_\tau $ to $H_{s,t}$ as
$$\g A = \g \left(\frac{1}{2}, \frac{b^2+a^2-a}{2b}\right)= \left(\frac{1}{2},\frac{1}{2}\right)$$
$$\g B =  \g \left(a-\frac{1}{2}, \frac{b^2 - a^2 +a}{2b}\right) = \left(a-\frac{1}{2}, \frac{1}{2}\right)$$
$$ \g C = \g \left(-\frac{1}{2},  \frac{b^2+a^2-a}{2b}\right)= \left(-\frac{1}{2},\frac{1}{2}-\frac{a}{a^2+b^2}\right)$$

\end{proof}

\begin{corollary}\label{cor:prob}
$P(\tau) = F(H(s,t),H(s,t),H(s,t))/area(H(s,t))^2$
\end{corollary}

 \section{Computations}



Fix $s,t \in [0,1/2]$ and let $H=H(s,t) $. We denote by $T$ (resp. $-T$)  the upper left  (resp. the bottom right) triangle. Let $Q = H \cup T \cup (-T)$ be the unit square containing $H$.
 
\begin{lemma}\label{lem:sep}
$F(H,H,H) = F(Q,Q,Q) -4F(H,H,T)  - 2F(H,T,T) -2F(T,Q,Q).$
\end{lemma}
 
 \begin{proof}
$$ F(Q,Q,Q) = F(H,Q,Q) +2F(T,Q,Q)$$ 
By symmetry of $Q$, $T$, and  $-T$. 
$$F(H,Q,Q) = F(H,H,Q) + 2F(H,T,Q)$$
by symmetry of $H$, $T$ and $-T$. Similarly,
$$F(H,H,Q)= F(H,H,H) +2F(H,H,T)$$
$$F(H,T,Q) = F(H,T,-T) + F(H,T,T)+F(H,T,H).$$
Therefore : 
\begin{eqnarray*}
F(Q,Q,Q) &=& F(H,H,H)+ 2F(H,H,T) +\\
&+& 2(F(H,T,-T) + F(H,T,T)+F(H,T,H))+ 2 F(T,Q,Q)\\
F(Q,Q,Q) &= &F(H,H,H) + 4F(H,H,T)  + 2F(H,T,T) +2F(T,Q,Q)
\end{eqnarray*}
Since $F(H,H,T)=F(H,T,H)$ and $F(H,T,-T)=0$, where the last equality follows from the fact that $(-T) \cap T_w = \emptyset$ for all $w \in H$.

 \end{proof}
 
 \begin{lemma}\label{lem:TQQ}
$F(T,Q,Q)= \frac{st}{24}\left(3 + 2 s + 2 t + s t \right).$
\end{lemma}
\begin{proof}
We can assume, by symmetry across the vertical axis, that $T$ is in the first quadrant. For $(u,v)\in T$, $Q \cap  Q_{(u,v)}$ is a rectangle with sides of length $(1-u)$ and $(1-v)$. Therefore 
$$F(T,Q,Q) = \int_{(u,v)\in T} (1-u)(1-v) \, du dv.$$
$T$ is parametrized  by $u \in [1/2 - s,1/2]$ and $v\in \left[\frac{1}{2} - \frac{t}{s}\left(u- \frac{1}{2} + s \right), 1/2\right].$
Therefore 
\begin{align*}
F(T,Q,Q) &= \int_{u=1/2-s}^{1/2} \int_{v= \frac{1}{2} - \frac{t}{s}\left(u- \frac{1}{2} + s \right)}^{1/2}  (1-u)(1-v) \, d u dv\\
	&= \int_{x=0}^{s}\int_{y=0}^{  \frac{t}{s}\left(s-x\right)}\left(\frac{1}{2}+x\right)\left(\frac{1}{2}+y\right) \,d x d y
\end{align*}
 where we make a change of variables $x = 1/2-u, y = 1/2-v$. Computing further,
 \begin{align*}
F(T,Q,Q) & = \int_{x = 0}^s \frac{1}{2}\left(\frac{1}{2}+x\right)\left(\left(\frac{1}{2}+\frac{t}{s}\left(s-x\right)\right)^2 - \frac{1}{4}\right) \, dx\\
	      & =  \frac{t}{4 s^2} \int_{x = 0}^s \left(1+ 2 x\right)\left(s-x\right)\left(s(1+t) - tx\right) \, dx\\
	     & =  \frac{t}{4 s^2} \int_{x = 0}^s 2 t x^3 + \left( t - 2s - 4 st\right) x^2 - s \left(1 + 2t - 2s - 2s t\right) x + s^2(1+t)  \, dx \\
	     & =  \frac{t}{4 s^2} \left( \frac{t s^4}{2} + \frac{\left(t - 2s - 4 st\right) s^3}{3} - \frac{\left(1 + 2t - 2s - 2s t\right)s^3}{2}+ s^3(1+t)\right)\\
	    & =  \frac{st}{24} \left( 3 s t  + 2 \left(t - 2s - 4 st\right) - 3\left(1 + 2t - 2s - 2s t\right)+ 6 (1+t)\right)\\
	    & = \frac{st}{24}\left(3 + 2 s + 2 t + s t\right)
\end{align*}

\end{proof}

For the remaining computations we will need to introduce the hexagon $V = V(s,t) = \{ u - v \mid u, v \in T \}$. Note that $V$ is the set of all vectors $w$ such that $T \cap T_w \neq \emptyset$. It is not hard to see that $V$ is the hexagon depicted in Figure \ref{fig:V_st}.

\begin{figure}[hbt]
    \begin{center}
\begin{overpic}[scale=1]{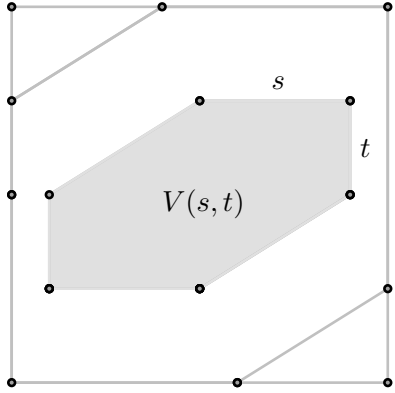}
\put(67,74){$s$}
\put(88,58){$t$}
\put(41,45){$V(s,t)$}
\end{overpic}
\end{center}
\caption{The hexagon $V = V(s,t)$ of all vectors $w$ such that $T \cap T_w \neq \emptyset$.}
\label{fig:V_st}
\end{figure}

 \begin{lemma}\label{lem:HTT}
$ F(H,T,T)=area(T)^2 = \frac{s^2t^2}{4}$
\end{lemma}
\begin{proof}
Notice that $V(s,t) \subseteq H(s,t)$ for all $s,t \in [0,1/2]$. Therefore $F(H,T,T) = F(V,T,T) = F(\R^2,T,T) = area(T)^2$ by (5) of Proposition \ref{prop:prop_of_F}.

\end{proof}

 \begin{lemma}\label{lem:HHT}
 $F(H,H,T)= \frac{st}{24} \left( 3 + 2s + 2t - 11 st\right)$
\end{lemma}

\begin{proof} For this computation, we will divide $H$ into 7 disjoint regions $H = \bigcup_{i = 0}^6 P_i$. Consider the line in the $u,v$-plane given by $v = t \, u / s$. Let $$P_0 = \{ (u,v) \in H \mid (v > t \, u / s) \vee ( (u > s) \wedge (v > t) )  \vee ((u < -s)   \wedge  (v < -t))\}$$ 
$$P_1 = \{ (u,v) \in H \mid (u > 0) \wedge (v < 0) \wedge (v + t < t \, u /s) \}$$
$$P_2 =  (s,1/2) \times (0, t)$$
$$P_3 = (-s,0) \times (-t, -1/2)$$
$$P_4 =  \{ (u,v) \in H \mid (0 < u < s) \wedge (0 < v < t) \wedge (v < t \, u /s)  \}$$
$$P_5 = \{ (u,v) \in H \mid (-s < u < 0) \wedge (-t < v < 0) \wedge (v < t \, u /s)  \}$$
$$P_6 = \{ (u,v) \in H \mid (0 < u < s) \wedge (0 < v < t) \wedge (v + t > t \, u /s)  \}$$
Here, $\vee$ denotes ``or'' and $\wedge$ denotes ``and.'' You can see the labeled regions on Figure \ref{fig:HHT}.

\begin{figure}[hbt]
    \begin{center}
\begin{overpic}[scale=1.5]{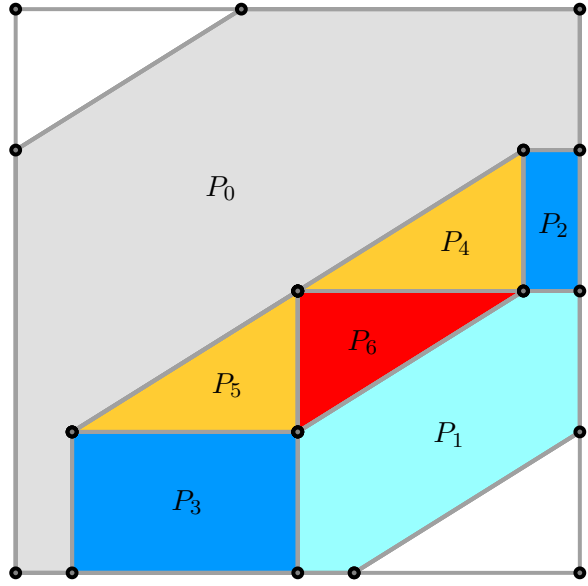}
\put(35,65){$P_0$}
\put(72,25){$P_1$}
\put(88.9,59){$P_2$}
\put(29.5,14){$P_3$}
\put(73,56){$P_4$}
\put(36,33){$P_5$}
\put(58,40){$P_6$}
\end{overpic}
\end{center}
\caption{Partition of $H$ into seven region for computations in Lemma \ref{lem:HHT}.}
\label{fig:HHT}
\end{figure}
\begin{itemize}
\item ($P_0$) It is easy to see that for $w \in P_0$ one has $H \cap T_w = \emptyset$, so $F(P_0, H, T) = 0$.
\item ($P_1$) For $w \in P_1$,  one has $T_w \subset H$. Therefore,  $$F(P_1, H, T) = area(T) \cdot area(P_1) = area(T) \left(\frac{1}{4} - 2 \, area(T)\right) = \frac{st}{8}\left(1 - 4 st\right) .$$
\item ($P2$ and $P3$) Note that $F(P_2,H,T)$ and $F(P_3,H,T)$ are symmetric up to a switch of $s$ and $t$. Let us compute $F(P_2, H, T)$. For very $(u,v) \in P_2$, the intersection $H \cap T_{(u,v)}$ is triangle similar to $T$. In fact, the height of this triangle is exactly $t - v$. Thus, $$F(P_2, H, T) = \int_{u = s}^{1/2} \int_{v = 0}^t \frac{st}{2} \frac{(t-v)^2}{t^2} \, du dv = \frac{st^2}{6}\left(\frac{1}{2} - s\right) = \frac{st^2}{12}\left(1 - 2s\right).$$
By flipping $s$ and $t$, we obtain $$F(P_3, H, T) = \frac{s^2t}{12}\left(1 - 2t\right).$$
\item ($P4$ and $P5$). Similarly, $F(P_4,H,T)$ and $F(P_5,H,T)$ are symmetric up to a switch of $s$ and $t$. Let us compute $F(P_4, H, T)$. For $(u,v) \in P_4$, $area(H \cap T_{(u,v)})$ is the area of a similar triangle of height $t-v$ {\it minus} $area(T \cap T_{u,v})$, as seen in Figure \ref{fig:T_cap_H_1}.
\begin{figure}[hbt]
    \begin{center}
\begin{overpic}[scale=1]{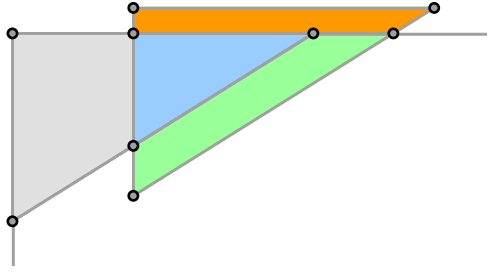}
\end{overpic}
\end{center}
\caption{Partition of $H \cap T_{(u,v)}$ for $(u,v) \in P_4$. The orange part is outside of $H$.}
\label{fig:T_cap_H_1}
\end{figure}
Therefore, 
 \begin{align*}
F(P_4, H, T) &= \int_{u = 0}^s \int_{v = 0}^{t \, u /s}  \frac{st}{2} \frac{(t-v)^2}{t^2} \, du dv  - \int_{(u,v) \in P_4} area(T \cap T_{u,v}) \, du dv\\
		& = \frac{st^2}{6} \int_{u = 0}^s 1 - \left(1 - \frac{u}{s}\right)^3 \, du - F(P_4, T,T)\\
		& = \frac{st^2}{6}\left(s - \frac{s}{4}\right) - \frac{1}{6} \, area(T)^2 =  \frac{st^2}{6}\left(s - \frac{s}{4}\right) - \frac{s^2 t^2}{24} = \frac{s^2t^2}{12}.
\end{align*}
We observe that $F(P_4, T,T) = \frac{1}{6} \, area(T)^2$ as follows. Apply $A \in \mathrm{SL}_2(\R)$ that sends $T$ to an equilateral triangle. Then $A \cdot P_4$ is a sixth of the regular hexagon $A \cdot V$. By symmetry, $F(P_4, T,T) = \frac{1}{6} F(V,T,T) = \frac{1}{6} \, area(T)^2$. Finally, switching $s$ and $t$, $$F(P_5, H, T) =  \frac{s^2t^2}{12}.$$

\item ($P_6$) For $(u,v) \in P_6$, observe that $area(H \cap T_{(u,v)}) = area(T) - area(T \cap T_{(u,v)})$ as in Figure \ref{fig:T_cap_H_2}.
\begin{figure}[hbt]
    \begin{center}
\begin{overpic}[scale=1]{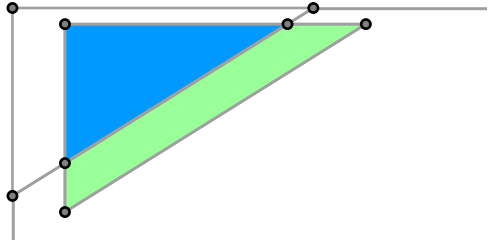}
\end{overpic}
\end{center}
\caption{Partition of $H \cap T_{(u,v)}$ for $(u,v) \in P_6$.}
\label{fig:T_cap_H_2}
\end{figure}
\end{itemize}
We can therefore compute
 \begin{align*}
F(P_6, H, T) & = \int_{(u,v) \in P_6} area(T) - area(T \cap T_{(u,v)}) \, du dv = area(T)^2 - F(P_6, T,T)\\
                    & = \frac{s^2t^2}{4} - \frac{s^2t^2}{24} = \frac{5\, s^2 t^2}{24}.
\end{align*}

We can now combine all of our computations to obtain :
\begin{align*}
F(H,H,T) &= \sum_{i = 0}^6 F(P_i,H,T)\\
	  &= 0 + \frac{st}{8}\left(1 - 4 st\right) +  \frac{st^2}{12}\left(1 - 2s\right)+\frac{s^2t}{12}\left(1 - 2t\right) + \frac{s^2t^2}{12} + \frac{s^2t^2}{12} + \frac{5\, s^2 t^2}{24}\\
	  & = \frac{1}{24}\left( 3st(1-4st) + 2st^2(1-2s) + 2 s^2t(1-2t) + 9 s^2 t^2 \right)\\
	  & = \frac{1}{24}\left(3st + 2 st^2 + 2s^2 t - 11 s^2t^2\right) = \frac{st}{24} \left( 3 + 2s + 2t - 11 st\right).
\end{align*}
\end{proof}

\begin{lemma} $F(H,H,H) = \frac{9}{16} - \frac{st}{12} \left(9 + 6 s + 6 t - 15 st\right) =  \frac{9}{16} - \frac{3st}{4} - \frac{s^2t}{2} - \frac{st^2}{2} + \frac{5 s^2t^2}{4}.$

\end{lemma}

\begin{proof} We now combine Lemmas \ref{lem:sep}, \ref{lem:HHT}, \ref{lem:HTT}, \ref{lem:TQQ} and Proposition \ref{prop:prop_of_F} (4) to compute : 
\begin{align*}
F(H,H,H) &= F(Q,Q,Q) -4F(H,H,T)  - 2F(H,T,T) -2F(T,Q,Q)\\
	       & = \frac{9}{16} - \frac{st}{6} \left( 3 + 2s + 2t - 11 st\right) - \frac{s^2t^2}{2} - \frac{st}{12}\left(3 + 2 s + 2 t + s t \right)\\
	       & = \frac{9}{16} - \frac{st}{12} \left( 6 + 4s + 4t - 22 st + 6 st + 3 + 2s + 2t + st\right)\\
	       & = \frac{9}{16} - \frac{st}{12} \left(9 + 6 s + 6 t - 15 st\right)\\
	       & = \frac{9}{16} - \frac{3st}{4} - \frac{s^2t}{2} - \frac{st^2}{2} + \frac{5 s^2t^2}{4}.
\end{align*}

\end{proof}

{\it Proof of Theorem \ref{thm:flat_torus}} By construction, $area(H) = (1- s t)$. For $\tau = a + ib \in \mathscr{D}$ with $a \in [0,1/2]$, Corollary \ref{cor:prob} gives $$P(\tau) = F(H,H,H)/area(H)^2 = \frac{1}{(1- st)^2} \left( \frac{9}{16} - \frac{3st}{4} - \frac{s^2t}{2} - \frac{st^2}{2} + \frac{5 s^2t^2}{4} \right).$$
By Proposition \ref{prop:st_shape},  $s= a$ and $t= \frac{a}{a^2+b^2}$, so $1-st = \frac{b^2}{a^2+b^2}$. Computing, 
\begin{align*}
P(\tau)  & = \frac{(a^2+b^2)^2}{b^4}  \left(\frac{9}{16} - \frac{3 a^2}{4(a^2+b^2)} - \frac{a^3}{2(a^2+b^2)} - \frac{a^3}{2(a^2+b^2)^2} + \frac{5 a^4}{4(a^2+b^2)^2}\right)\\
& = \frac{1}{16 b^4}  \left(9(a^2+b^2)^2 -12 a^2(a^2+b^2) - 8 a^3(a^2+b^2) - 8a^3 + 20 a^4\right)\\
& =  \frac{1}{16 b^4}\left((9-12+20)a^4+9b^4 + (18 -12)a^2b^2-8a^5-8a^3b^2-8a^3 \right)\\
& = \frac{9}{16} + \frac{3a^2}{8b^2} - \frac{a^3}{2b^2} - \frac{a^3}{2b^4} + \frac{17 a^4}{16b^4} - \frac{a^5}{2b^4}.
\end{align*}
Since $D_\tau$ and $D_{-\overline{\tau}}$ are mirror images of each other across the vertical axis, the above formula holds for all $\tau = a + i b \in \mathscr{D}$ if $a$ is replaced by $|a|$.

\qed


\bibliographystyle{alpha}

 \end{document}